\documentclass[12pt]{amsart}
\usepackage{color,graphicx,array, amssymb, amscd,slashed}
\newtheorem{Theorem}{Theorem}[section]

\newtheorem{Proposition}[Theorem]{Proposition}
\newtheorem{Corollary}[Theorem]{Corollary}
\theoremstyle{definition}
\newtheorem{Definition}{Definition}
\theoremstyle{remark}
\newtheorem{Remark}[Theorem]{Remark} 
\numberwithin{equation}{section}
\newcommand{\R}{\mathbb R}
\newcommand{\C}{\mathbb C}

\newcommand{\D}{\mathbb D}
\newcommand{\Q}{\mathcal Q}
\newcommand{\RP}{\mathbb P^3}
\newcommand{\SL}{{\rm SL}_4 \mathbb R}
\renewcommand{\sl}{{\rm sl}_4 \mathbb R}
\newcommand{\SLC}{{\rm SL}_4 \mathbb C}
\newcommand{\GKO}{\SL/K_2}
\newcommand{\GKSO}{\SL/{\rm SO}_{2, 2}}
\newcommand{\GKOT}{\SL/K_1}
\newcommand{\GKS}{\SLC/K}
\newcommand{\GKOTC}{\SLC/K_1^{\mathbb C}}
\newcommand{\gs}{g_2}
\newcommand{\gf}{g_1}
\newcommand{\js}{J_2}
\newcommand{\jf}{J_1}

\newcommand{\LSLO}{\Lambda {\rm SL}_4 \mathbb R_{\tau_2}}
\newcommand{\LSLT}{\Lambda {\rm SL}_4 \mathbb R_{\kappa}}
\newcommand{\ad}{\operatorname{Ad}}
\newcommand{\di}{\operatorname{diag}}
\newcommand{\offdi}{\operatorname{offdiag}}
\newcommand{\tr}{\operatorname{Tr}}
\begin{document}
\title{A loop group method for Demoulin surfaces 
 in the $3$-dimensional real projective space}
\author[S.-P.~Kobayashi]{Shimpei Kobayashi}
 \address{
 Graduate School of Science and Technology,
 Hirosaki University,
 Hirosaki, 036-8561, Japan}
 \email{shimpei@cc.hirosaki-u.ac.jp}
 \thanks{The author is partially supported by Kakenhi 
 23740042}
\date{\today}
\pagestyle{plain}
\begin{abstract} 
 For a surface in the $3$-dimensional real projective space,
 we define a Gauss map, 
 which is a quadric in $\mathbb R^{4}$ and 
 called the first-order Gauss map. 
 It will be shown that the surface is a Demoulin surface 
 if and only if the first-order Gauss map is conformal, 
 and the surface is a projective minimal 
 coincidence surface or a Demoulin surface  
 if and only if the first-order Gauss map is harmonic.
 Moreover for a Demoulin surface, it will be shown that 
 the first-order Gauss map can be obtained by 
 the natural projection of the Lorentz primitive map into 
 a $6$-symmetric space.
 We also characterize Demoulin surfaces via 
 a family of flat connections 
 on the trivial bundle $\D \times \SL$ over a simply connected 
 domain $\mathbb{D}$ in the Euclidean $2$-plane. 
\end{abstract}
\maketitle
%%%%%%    TEXT START    %%%%%%%%
\section*{Introduction}
 Curves and surfaces in the $3$-dimensional 
 real projective space $\RP$ were the central theme 
 of differential geometry in 19th century.
 Especially, various transformations for a surface in $\RP$ were 
 introduced by Darboux, Demoulin, Titzeica, Godeaux, Rozet,
 Wilczynski, etc., and their properties were extensively studied. 
 The most prominent features of the theory of transformations 
 were the {\it Laplace sequence}  and the {\it line/sphere congruences}
 of a surface. It is well known that Toda equations which 
 discovered in the theory of integrable systems in 1970s
 had been already known as the periodic Laplace sequence, and 
 the classical Darboux and B\"acklund transformations were 
  defined by sphere 
 congruences and tangential line congruences, respectively.

 The Demoulin surface is characterized by 
 the coincidence of general four Demoulin transformations 
 of a surface, which are given by the envelopes of Lie quadrics.
 On the one hand, the projective minimal surface
 is defined by a critical point of the projective area 
 functional. It is known that 
 Demoulin surfaces give a special class of projective 
 minimal surfaces. Moreover, using Pl\"ucker embedding from 
 $\RP$ to $\mathbb P^5$,  Godeaux introduced an analogue 
 of the Laplace sequence, the so-called {\it Godeaux sequence}
 of a surface in $\mathbb P^5$. Then the surface is a Demoulin 
 surface if and only if the Godeaux sequence is six periodic. For 
 the modern treatment of the subject, 
 we refer the readers to \cite{Sa}.  
 
 By using modern theory of integrable systems   
 and differential geometry of harmonic maps, projective 
 minimal surfaces and Demoulin surfaces were investigated 
 in \cite{Fe, FS, BHJ}.
 More precisely in \cite{BHJ}, through Pl\"ucker 
 embedding from $\RP$ to $\mathbb P^5$ projective minimal surfaces 
 were characterized by Lorentz harmonicity of the conformal Gauss map, 
 which takes values in a certain indefinite Grassmannian.
 In \cite{FS}, Demoulin surfaces were characterized by 
  a certain Toda equation and the B\"acklund transformation 
 of a Demoulin surface was constructed. Moreover, many classes 
 of surfaces characterized by geometric properties were 
 related to various integrable systems in \cite{Fe}.
  
 In this paper, we study Demoulin surfaces via a loop group method.
 We first define a Gauss map for a surface in $\RP$,
 which is a quadric in $\R^4$, 
 and called the {\it first-order Gauss map}.
 The first-order Gauss map has the first-order contact 
 to the surface.
 It will be shown that the first-order Gauss map 
 is conformal if and only if the surface is a Demoulin surface, 
 see Proposition \ref{prop:conformality}.

 Then the Lorentz harmonicity
 of the first-order Gauss map is studied.
 It will be shown that 
 the first-order Gauss map is Lorentz harmonic if and only if 
 the surface is a Demoulin or a projective minimal 
 coincidence surface, see Theorem \ref{thm:Demoulin}. 
 We note that coincidence surfaces are 
 simple examples of a class of surfaces which have 
 nontrivial projective deformations, the so-called 
 {\it projective applicable surfaces}.
 Since the target space of the first-order Gauss map 
 is a symmetric space, the Lorentz harmonic map is also characterized 
 by a family of flat connections on the trivial bundle 
 $\D \times  \SL$. 

 Combining 
 the results in Proposition \ref{prop:conformality} and Theorem 
 \ref{thm:Demoulin}, we see that the first-order Gauss map 
 is conformal Lorentz harmonic if and only if 
 the surface is a Demoulin surface, 
 see Corollary \ref{coro:Demoulin}.
 Finally it will be shown that 
 the Gauss map of a Demoulin surface can be obtained 
 by the natural 
 projection of the Lorentz primitive map into a $6$-symmetric space, 
 see Theorem \ref{coro:primitivity}.

 In Appendix \ref{ap:confGauss}, we review results of Thomsen 
 \cite{Th}, that is, a surface is projective minimal if and only 
 if the conformal Gauss map is conformal Lorentz harmonic. 
 In \cite[Theorem 7]{BHJ}, 
 by using Pl\"ucker embedding 
 from $\RP$ to $\mathbb P^5$, the conformal Gauss map 
 can be considered as the map into a certain indefinite Grassmannian 
 in $\R^6$ and the Thomsen's theorem was proved.
 Theorem \ref{thm:projmini} is another reformulation of it.

\textbf{Acknowledgements:} 
 The author would like to express his sincere thanks to 
 Prof. T.~Sasaki and Prof. J.~Inoguchi for their helpful 
 discussion and comments on the draft of this paper.
 The author also would like to express his sincere thanks to 
 the anonymous referee for his/her careful reading and 
 critical comments for the original version of this paper.
\section{Preliminaries}
\subsection{Surfaces in $\RP$ and the Wilczynski frames}
 The {\it canonical system} of a surface $S$ in 
 the $3$-dimensional real projective space $\RP$ 
 is given as follows, \cite{Wil}, \cite[Section 2.2]{Sa}:
 \begin{equation}\label{eq:canonicalsystem}
 f_{x x} = b f_y + p f, \;\;
 f_{y y} = c f_x + q f,
 \end{equation}
 where $f$ is a lift of $S$ in $\R^4 \setminus \{\mathbf 0\}$, 
 $b, c, p$ and $q$ are functions of 
 real variables $x$ and $y$, and 
 the subscripts $x$ and $y$ denote the partial derivative 
 with respect to $x$ and $y$, respectively. Let 
 $f=(f^0, f^1, f^2, f^3)^t \in \R^4 \setminus \{\mathbf  0\}$ and assume that $f^0 \neq 0$. 
 Then the surface $S$ is given by $S = \tfrac{1}{f^0}(f^1, f^2, f^3)^t$
 and a straightforward computation shows that
\begin{align*}
  S_{xx} = b S_y -2 (\log f^0)_x S_x \;\;\mbox{and}\;\;  
  S_{yy} = c S_x -2 (\log f^0)_y S_y.
\end{align*}
 This implies that $x$ and $y$ are asymptotic coordinates on $S$.
 Thus the coordinates $(x, y)$ induce the Lorentz structure on 
 the surface $S$.
 It is known that $8 bc \>dx dy$ is 
 an absolute invariant symmetric quadratic form, 
 which is called the {\it projective metric} 
 and $8bc$ is called the {\it Fubini-Pick invariant} 
 of a surface $S$. It is also known that the conformal 
 class of $b \>dx^3 + c \>dy^3$ is an absolute 
 invariant cubic form.
 It is known that a surface whose 
 Fubini-Pick invariant $8 b c =0$ is ruled, thus we 
 assume that $b c \neq 0$. 
 Then the 
 {\it Wilczynski frame}  is defined as follows:
 $$
 F= (f, f_1, f_2, \eta ),
 $$ 
 where 
\begin{align*}
 f_1 &= f_x  - \frac{c_x}{2 c} f, \;\; f_2 = f_y  - 
 \frac{b_y}{2 b} f, \\
 \eta &= f_{xy}  - \frac{c_x}{2 c} f_y - 
 \frac{b_y}{2 b} f_x + \left(\frac{b_y c_x}{4 b c}-\frac{bc}{2}\right)f.
\end{align*}
 Then a straightforward computation shows that 
 the Wilczynski frame $F$ satisfies the following equations:
\begin{equation}\label{eq:movingframe}
 F_x = F U \;\;\mbox{and} \;\;F_y = F V,
\end{equation}
 where 
\begin{equation}\label{eq:movingframeUV}
 U= 
 \begin{pmatrix}
 \frac{c_x}{2 c} & P & k  &
  b Q  \\[0.1cm]
 1  & -\frac{c_x}{2 c} & 0 & k \\[0.1cm]
 0&b&\frac{c_x}{2 c} & P \\[0.1cm]
 0 & 0 & 1 &-\frac{c_x}{2 c}
 \end{pmatrix}, \;\;
 V= 
 \begin{pmatrix}
 \frac{b_y}{2 b} & \ell & 
 Q  & c P \\[0.1cm]
 0  & \frac{b_y}{2 b} & c & Q \\[0.1cm]
 1&0& -\frac{b_y}{2 b} & \ell \\[0.1cm]
 0 & 1 & 0 &-\frac{b_y}{2 b}
 \end{pmatrix}.
\end{equation}
 Here we introduced functions $k, \ell, P$ and $Q$ of 
 two variables $x$ and $y$ as follows:
\begin{align}
 k &= \frac{bc  - (\log b)_{x y}}{2},\;\;\; 
 \ell = \frac{bc  - (\log c)_{xy}}{2}, \label{eq:kell}\\
 P &= p + \frac{b_y}{2} -\frac{c_{x x}}{2 c} + \frac{c_x^2}{4 c^2},
 \;\;\;Q  = q + \frac{c_x}{2} - \frac{b_{yy}}{2 b}+ \frac{b_y^2}{4 b^2}.
 \label{eq:QandP}
\end{align}
 The compatibility conditions of \eqref{eq:movingframe} are 
 \begin{align}
% (\log b)_{x y} = bc - 2 k, \; \; (\log c)_{x y} = b c - 2 \ell, \;  \\
  Q_x = k_y + k \frac{b_y}{b}, \;\; P_y = \ell_x +  \ell\frac{c_x}{c},
 \label{eq:comp1}\\
 b Q_y + 2 b_y Q = c P_x+ 2 c_x P. \label{eq:comp2}
 \end{align}
 These equations are nothing but the projective Gauss-Codazzi equations
 of a surface $S$. Since the traces of $U$ and $V$ are zero, 
 the Wilczynski frame 
 $F$ takes values in $\SL$ up to initial condition.
 From now on, we assume that the Wilczynski frame $F$ 
 takes values in $\SL$.

\begin{Remark}
 Instead of real coordinates $(x, y)$, one can 
 use the complex coordinates $(z, \bar z)$ with $z = x+ i y$. 
 Then  the induced conformal structure is Riemannian and 
 the following discussion is parallel to the case of 
 real coordinates. However, for simplicity, we consider 
 only the case of real coordinates.
\end{Remark}

\subsection{Projective minimal surfaces and Demoulin surfaces}
 It is known that the {\it projective minimal surface} 
 is defined by a critical point of the projective area functional:
 $$
  \int b c \> dx dy,
 $$
 where the functions $b$ and $c$ are defined 
 in \eqref{eq:canonicalsystem}. Then the projective minimality 
 can be computed as in \cite{Th}:
 \begin{equation}\label{eq:projmin}
 b Q_y + 2 b_y Q  =0 \;\;\mbox{and}\;\;
 c P_x + 2 c_x P =0,
 \end{equation}
 where the functions $P$ and $Q$ are defined in \eqref{eq:QandP}.

 The {\it Demoulin surface} is defined by  the coincidence of general 
 four Demoulin transformations of a surface, which are given 
 by the envelopes of Lie quadrics. It is known that 
 Demoulin surfaces are characterized by the functions 
 $P$ and $Q$ in \eqref{eq:QandP}, 
 see \cite[Definition 2.8]{Sa}:
 \begin{equation}\label{eq:Demoulin}
 P =Q=0.
 \end{equation}
 \begin{Remark}
 From the equations in \eqref{eq:projmin} and \eqref{eq:Demoulin},
 it is easy to see that 
 Demoulin surfaces are projective minimal surfaces.
 \end{Remark}

 \subsection{(Lorentz) Harmonic and (Lorentz) primitive maps into ($k$-)symmetric spaces}
 It is known that the loop group method 
 can be applied to harmonic maps from 
 surfaces into symmetric spaces, see \cite{BP, DPW}. 
 Let $M$ and $N$ be a Riemann (or Lorentz) surface and a semisimple 
 symmetric space, respectively 
 and $\varphi$ a map from $M$ into $N$.
 We denote the symmetric space $N$ as quotient $G/K$ with 
 semisimple Lie group $G$ and
 closed subgroup $K$ of $G$ such that 
 $(G_\sigma)_o \subseteq K \subset G_\sigma$, 
 where $(G_\sigma)_o$ is the identity component of 
 the fixed point group $G_\sigma$ 
 of the involution $\sigma$ of the symmetric space $N$. 
 Let $\Phi$ be the frame of $\varphi$ taking values in $G$ and
 $\alpha =\Phi^{-1} d \Phi$ the Maurer-Cartan form.
 According to  the eigenspace decomposition of $\mathfrak g$
 with respect to the derivative of $\sigma$, that is 
 $\mathfrak g = \mathfrak k \oplus \mathfrak p$, 
 we define $\alpha^{\lambda}$ as follows:
 $$
 \alpha^{\lambda} = \alpha_{\mathfrak k} + 
 \lambda^{-1} \alpha_{\mathfrak p}^{\prime} +
 \lambda \alpha_{\mathfrak p}^{\prime \prime},\;\; 
 \lambda \in \C^{\times},
 $$
 where $\alpha_{\mathfrak k}$ and $\alpha_{\mathfrak p}$
 denote the $\mathfrak k$- and $\mathfrak p$-parts, and 
 $\prime$ and $\prime \prime$ denote the $(1, 0)$- and 
 $(0, 1)$-parts, respectively.
 \begin{Remark}
 For a Riemann surface $M$ with conformal coordinates $z = x + iy$,
 the $(1, 0)$- and $(0, 1)$-parts denote $dz$ and $d \bar z$ parts, 
 respectively, and  for a Lorentz surface $M$ with null coordinates 
 $(x, y)$, the $(1, 0)$- and $(0, 1)$-parts denote $dx$ and $d y$ parts, 
 respectively. 
 \end{Remark}
 The following theorem is a fundamental fact about 
 (Lorentz) harmonic maps from surfaces into symmetric spaces, 
 see \cite{BP, DPW}.
\begin{Theorem}\label{thm:harmonicflat}
 Let $M$ be a Riemann (or Lorentz) surface and $N$ 
 a semisimple symmetric space.
 A map $\varphi : M \to N$ 
 is a (Lorentz) harmonic map if and only if 
 $d + \alpha^{\lambda}$ is a family of flat connections.
\end{Theorem}
 If the target manifold $N$ is a semisimple $k$-symmetric space ($k>2$), 
 then there does not exist a loop group formulation 
 for general (Lorentz) harmonic maps from a surface into $N$ 
 as the above. 
 Instead, we restrict our attention to a rather 
 special kind of (Lorentz) harmonic maps, 
 the {\it (Lorentz) primitive maps}
 in a $k$-symmetric space, so that the loop group formulation 
 can be applied.
\begin{Definition}
 Let $\varphi$ be a map from a Riemann or Lorentz surface $M$ into 
 a semisimple $k$-symmetric space $N =G/K$ with the order 
 $k$ automorphism $\sigma$ ($k>2$) and $\alpha = \Phi^{-1} d\Phi$ 
 the Maurer-Cartan form of the frame $\Phi$ of $\varphi$.
 Moreover, let 
 $\mathfrak g = \mathfrak g_0 
 \oplus \mathfrak g_1  \oplus \mathfrak g_2 \oplus \cdots \oplus
 \mathfrak g_{k-1}$ be the eigenspace decomposition of $\mathfrak g$
 and 
 $\mathfrak g^{\C} = \mathfrak g_0^{\C} 
 \oplus \mathfrak g_1^{\C}  \oplus \mathfrak g_2^{\C} 
 \oplus \cdots \oplus
 \mathfrak g_{k-1}^{\C}$ the eigenspace decomposition of the 
 complexification of $\mathfrak g$ 
 according to the derivative of $\sigma$ and define 
 $\mathfrak g_{i +kn} =\mathfrak g_i$ and 
 $\mathfrak g_{i +kn}^{\C} =\mathfrak g_i^{\C}$ 
 for $n \in \mathbb Z$.
 For the case of Riemann surface $M$, $\varphi$ is called the 
 {\it primitive map} if 
 \begin{equation}\label{eq:primitivedef}
 \mbox{$\alpha^{\prime}$ takes values in 
 $\mathfrak g_0^{\C} \oplus\mathfrak g_{-1}^{\C}$},
 \end{equation}
 where $\prime$ is the $(1, 0)$-part with respect to 
 the conformal structure on the Riemann surface $M$.
 For the case of Lorentz surface $M$, $\varphi$ is called the 
 {\it Lorentz primitive map} if 
 \begin{equation}\label{eq:Lorentzprimitivedef}
 \mbox{$\alpha^{\prime}$ takes values in 
 $\mathfrak g_0 \oplus\mathfrak g_{-1}$,
 and $\alpha^{\prime \prime}$ takes values in 
 $\mathfrak g_0 \oplus\mathfrak g_{1}$},
 \end{equation}
 where $\prime$ and $\prime \prime$ are the $(1, 0)$- and 
 $(0, 1)$-parts with respect to the conformal structure on 
 the Lorentz surface $M$, 
 respectively.
\end{Definition}
 The following is a basic fact about (Lorentz) primitive maps, 
 see \cite{Black}.
\begin{Proposition}
\mbox{}
\begin{enumerate}
\item A (Lorentz) primitive map into a semisimple $k$-symmetric space $N$
 ($k>2$) is (Lorentz) {\it equiharmonic}, that is, it is 
 (Lorentz) harmonic with respect to any 
 invariant metric on $N$.

\item
 Let $\varphi$ be a (Lorentz) 
 primitive map into a semisimple $k$-symmetric space $N=G/K$, ($k>2$),
 and $\pi :N \to G/H$ with $K \subset H$ 
 the homogeneous projection. Then 
 $\pi \circ \varphi$ is (Lorentz) equiharmonic.
\end{enumerate}
\end{Proposition}
 Let $\varphi$ be a primitive map 
 into a semisimple $k$-symmetric space $N \> (k>2)$
 and $\Phi$ the corresponding 
 frame. Moreover, let $\alpha$ be the Maurer-Cartan form of $\Phi$, 
 $\alpha = \Phi^{-1} d \Phi$. Define $\alpha^{\lambda}$ as follows:
 \begin{equation*}\label{eq:primitivealpha}
 \alpha^{\lambda} = \alpha_{0}
 +\lambda^{-1} \alpha_{-1}^{\prime} 
 +\lambda \alpha_{1}^{\prime \prime}, \;\;\lambda \in \C^{\times},
 \end{equation*}
 where $\alpha_{j}$
 is the $j$-th eigenspace of the derivative of $\sigma, \; 
 (j = -1, 0, 1)$. 
 The following is a well known fact, see for example, \cite{BP}.
\begin{Theorem}\label{thm:primitiveflat}
 Let $M$ be a Riemann (or Lorentz) surface and 
 $N$ a semisimple $k$-symmetric space ($k>2$).
 If a map $\varphi : M \to N$ is a  (Lorentz) primitive map 
 then $d + \alpha^{\lambda}$ is a family of flat connections.
\end{Theorem}

\section{Projective minimal coincidence surfaces 
and Demoulin surfaces}

\subsection{The first-order Gauss map}\label{subsc:Gaussmap}
 Let us use the following notation:
 $$
  \di (a, b, c, d) = \begin{pmatrix} 
 a &&& \\&b&&\\ &&c&\\  &&&d 
 \end{pmatrix}, \;\;\;
 \offdi (a, b, c, d) = \begin{pmatrix} 
  &&&a \\&& b&\\ & c& &\\  d&&& 
 \end{pmatrix}.
 $$
 Let $S$ be a surface in $\RP$ and $F$ the corresponding 
 Wilczynski frame defined in \eqref{eq:movingframe}.
 We first define a map $\gf$ by 
 $$
 \gf = F \jf F^t, 
 $$
 where $\jf = \offdi (1, 1, 1, 1)$.
 It is easy to see that $\gf$ maps to the space of 
 symmetric matrices with determinant one and signature $(2, 2)$, 
 which we denote by $\Q$. The special linear group 
 $\SL$ transitively acts on this space by $g P g^t \in \Q$ 
 with $g \in \SL$ and $P \in \Q$. Then 
 the point stabilizer at $\jf \in \Q$ is given by
 $K_1 = \{ X \in \SL \;|\; X \jf X^t = \jf\}$, which 
 is isomorphic to the special orthogonal group with 
 signature $(2, 2)$, which is denoted by ${\rm SO}_{2, 2}$. 
 Thus $\Q$ is 
 isomorphic to the symmetric space $\GKSO$:
 \begin{equation}\label{eq:Gaussmap1}
 \gf : M \to \Q \cong \GKOT = \GKSO.
 \end{equation}
 This map $\gf$ is known to be a quadric which has 
 the first order contact to the surface.
 Note that $\gf$ does not have the
 second order contact, see \cite[Section 22]{Lane}. 
 We call $\gf$ the {\it first-order Gauss map} for a surface $S$
 in $\RP$. 
 We now characterize the Demoulin surface by the first-order 
 Gauss map.
\begin{Proposition}\label{prop:conformality}
 The first-order Gauss map $\gf$ is conformal if and only if
 the surface $S$ is a Demoulin surface.
\end{Proposition}
\begin{proof}
 We first introduce the inner product on the tangent space of $\Q$ 
 as follows:
 $$
 \langle X, Y\rangle_{p} = \tr (p^{-1} Xp^{-1} Y), 
 \;\; X, Y \in T_{p} \Q,
 $$
 where $p$ is a  symmetric matrix of determinant one with 
 signature $(2,2)$.
 This inner product is invariant under the action of $g \in \SL$, 
 since 
 $$
 \langle gXg^t , g Y g^t\rangle_{gp g^t} = 
 \tr ((gp g^t)^{-1} gXg^t (gp g^t)^{-1} gY g^t) 
 =  \langle X, Y\rangle_{p}.
 $$
 A direct computation shows that 
 $$
 {\gf}_{x}= 2 F
 \begin{pmatrix}
 b Q & k & P & 0 \\ 
 k & 0 & 0 &1 \\
 P & 0 & b & 0 \\
 0 & 1 & 0 & 0
 \end{pmatrix}
 F^t,
\;\;
 {\gf}_{y}= 2 F
 \begin{pmatrix}
 c P & Q & \ell & 0 \\ 
 Q & c & 0 &0 \\
 \ell & 0 & 0 & 1 \\
 0 & 0 & 1 & 0
 \end{pmatrix}F^t.
 $$
 Thus 
 $$
\langle {\gf}_x,{\gf}_x \rangle = 16 P, \; 
\langle  {\gf}_y,  {\gf}_y \rangle  = 16 Q
\;\; \mbox{and} \;\;
\langle  {\gf}_x,  {\gf}_y \rangle = \langle {\gf}_y,{\gf}_x \rangle = 
 8(k + \ell)+ 4 bc.
$$
 Since  the coordinates $(x, y)$  are null for the conformal 
 structure induced by $S$,
 the first-order Gauss map $\gf$ is conformal if and only if $P = Q =0$.
\end{proof}

\subsection{Projective minimal coincidence surfaces  and 
 Demoulin surfaces}
 Let $\tau_1$ be the outer involution on 
 $\SL$ associated to $\Q$ in \eqref{eq:Gaussmap1}
 defined by $\tau_1 (X) = \jf X^{t-1} \jf, \; X \in \SL$ and 
 $\jf=\offdi (1, 1, 1, 1)$. 
 Abuse of notation, we denote the 
 differential of $\tau_1$ by the same symbol $\tau_1$  
 which is an outer involution on $\sl$:
\begin{equation}\label{eq:tau2}
\tau_1 (X) = - \jf X^t \jf, \; \;X \in \sl.
\end{equation}
 Let us consider the eigenspace decomposition
 of $\mathfrak g = \sl$ with respect to $\tau_1$, that 
 is, $\mathfrak g = \mathfrak k_1 \oplus \mathfrak p_1$,
 where $\mathfrak k_1$ is the $0$th-eigenspace 
 and $\mathfrak p_1$ is the $1$st-eigenspace as follows:
$$
 \mathfrak k_1 =\left\{\left.
\begin{pmatrix}
 a_{11} & a_{12} & a_{13} & 0\\
 a_{21} & a_{22} & 0 &  -a_{13}\\
 a_{31} & 0 & -a_{22} & -a_{12}\\
 0 & -a_{31} & -a_{21} & -a_{11}
\end{pmatrix}
\right| \;a_{ij} \in \R
\right\}, \;\; 
 \mathfrak p_1 =
\left\{\left.
\begin{pmatrix}
 a_{11} & a_{12} & a_{13} & a_{14}\\
 a_{21} & -a_{11} & a_{23} &  a_{13}\\
 a_{31} & a_{32} & -a_{11} & a_{12}\\
 a_{41} & a_{31} & a_{21} & a_{11}
\end{pmatrix}
\right| \;a_{ij} \in \R
\right\}.
$$ 
 According to this decomposition 
 $\mathfrak g = \mathfrak k_1 \oplus \mathfrak p_1$,
 the Maurer-Cartan form $\alpha = F^{-1} d F = U dx + Vdy$
 can be decomposed into  
 $$
 \alpha = \alpha_{\mathfrak k_1} + \alpha_{\mathfrak p_1} =
 U_{\mathfrak k_1}dx + V_{\mathfrak k_1}dy + U_{\mathfrak p_1}dx + 
 V_{\mathfrak p_1}dy,
 $$
 where $U = U_{\mathfrak k_1} + U_{\mathfrak p_1}$ and
 $V = V_{\mathfrak k_1} + V_{\mathfrak p_1}$.
 Let us insert the parameter $\lambda \in \mathbb R^{\times}$ 
 into $U$ and $V$ as follows:
 $$
 U^{\lambda} = U_{\mathfrak k_1} + \lambda^{-1} U_{\mathfrak p_1} \;\;\mbox{and}  \;\;
 V^{\lambda} = V_{\mathfrak k_1} + \lambda V_{\mathfrak p_1}.
 $$
 Then a family of $1$-forms $\alpha_{\lambda}$ 
 is defined as follows:
\begin{equation}\label{eq:alphalambda2} 
 \alpha^{\lambda} = 
 \alpha_{\mathfrak k_1} + \lambda^{-1} \alpha_{\mathfrak p_1}^{\prime}
 + \lambda \alpha_{\mathfrak p_1}^{\prime \prime} 
 = U^{\lambda} dx + V^{\lambda} dy.
\end{equation}
 In fact the matrices $U^{\lambda}$ 
 and $V^{\lambda}$ are explicitly given 
 as follows:
\begin{equation}\label{eq:extendedUV2}
 U^{\lambda}= 
 \begin{pmatrix}
 \frac{c_x}{2 c} & \lambda^{-1} P & \lambda^{-1} k  &  \lambda^{-1}b Q  \\[0.1cm]
 \lambda^{-1}  & -\frac{c_x}{2 c} & 0 & \lambda^{-1} k \\[0.1cm]
 0&\lambda^{-1} b&\frac{c_x}{2 c} & \lambda^{-1} P \\[0.1cm]
 0 & 0 & \lambda^{-1} &-\frac{c_x}{2 c}
 \end{pmatrix}, \;\;
 V^{\lambda}= 
 \begin{pmatrix}
 \frac{b_y}{2 b} & \lambda \ell &  \lambda Q  & \lambda c P \\[0.1cm]
 0  & \frac{b_y}{2 b} & \lambda c & \lambda Q \\[0.1cm]
 \lambda &0& -\frac{b_y}{2 b} & \lambda \ell \\[0.1cm]
 0 & \lambda  & 0 &-\frac{b_y}{2 b}
 \end{pmatrix}.
\end{equation}
 The following is the main theorem in this paper.
\begin{Theorem}\label{thm:Demoulin}
 Let $S$ be a surface in $\RP$ and $\gf$ the first-order
 Gauss map defined  in \eqref{eq:Gaussmap1}.
 Moreover, let $\alpha^{\lambda} \> (\lambda
 \in \R^{\times})$  be a family of $1$-forms defined in 
 \eqref{eq:alphalambda2}. Then 
 the following are mutually equivalent:
\begin{enumerate}
\item The surface $S$ is a Demoulin surface or a projective minimal 
 coincidence surface. 
\item The first-order Gauss map $\gf$ 
 is a Lorentz harmonic map into $\Q$.
\item $d + \alpha^{\lambda}$ is a family of flat connections 
 on $\D \times \SL$.
\end{enumerate}
\end{Theorem}
\begin{proof}
 Let us compute the flatness conditions of $d + 
 \alpha^{\lambda}$, 
 that is, the Maurer-Cartan equation $ d \alpha^{\lambda} + 
 \tfrac{1}{2}[\alpha^{\lambda} \wedge \alpha^{\lambda}]=0$. 
 A straightforward computation shows that these are equivalent to 
 \begin{align*}
 Q_x = P_y =0,& \;\; k_y +k  \frac{b_y}{b} =0, \;
 \ell_x + \ell \frac{c_x}{c} =0, \; \\
 b Q_y + 2  b_y Q =0, &\;\;
 c P_x + 2 c_x P =0.
 \end{align*}
 The surfaces with $P=Q =0$ satisfies the above equations and 
 they are Demoulin surfaces by \eqref{eq:Demoulin}. 
 Assume that $P\neq 0$ (The case of $Q\neq 0$ is similar).
 From the first equation and the last equation, 
 $P$ and $(\log c)_x$ depend only on $x$.
 Moreover from the equation  $\ell_x + \ell (\log c)_x =0$
 and the definition $\ell$ in \eqref{eq:kell}, $(\log b)_x =
 -2 (\log c)_x$ and thus $(\log b)_x$ depends only on $x$. 
 Thus $(\log b/c)_{xy} =0$, which means that 
 it is an {\it isothermally asymptotic surface}.
 Using a scaling transformation and a change of coordinates, 
 we can assume that 
 $b = c$. Then $\ell = k$ and the equations 
 $\ell_x + \ell (\log c)_x =0$ and $k_y + k (\log b)_y =0$
 imply that $b(=c)$ is constant. Thus $P$ and $Q$ are 
 constant, and from \eqref{eq:QandP}
 $p\neq 0$ and $q$ are constant. Therefore, the canonical system 
 is given by
 $$
 f_{xx} = f_y + p f, \;\;\; f_{yy} = f_x + q f.
 $$ 
 A surface satisfying the above equation 
 is the special case of the {\it coincidence surface}, 
 \cite[Example 2.19]{Sa}. In fact, it is easy to see 
 that the surface is a projective minimal coincidence surface.
 Thus the quivalence of $(1)$ and $(3)$ follows.

 The equivalence of $(2)$ and $(3)$ follows from Theorem
 \ref{thm:harmonicflat}, since the family of 
 $1$-forms $\alpha^{\lambda}$ is given 
 by the involution $\tau_1$ and it defines 
 the symmetric space $\Q = \GKOT$.
\end{proof}
\begin{Remark}
 Let $F^{\lambda}$ be a family of 
 frames such that $(F^{\lambda})^{-1} 
 d F^{\lambda} = \alpha^{\lambda}$. 
 It is easy to see from the forms of $U^{\lambda}$
 and $V^{\lambda}$ in \eqref{eq:extendedUV2} 
 that $F^{\lambda}$ is not 
 the Wilczynski frame of a Demoulin surface or projective 
 minimal coincidence surface except $\lambda =1$. 
 However conjugating 
 $F^{\lambda}$ by $DF^{\lambda}D^{-1}$ with 
 $D=\di (1, \lambda, \lambda^{-1}, 1)$, 
 the frames $DF^{\lambda}D^{-1}$ give 
 a family of Wilczynski frames for Demoulin surfaces
  or projective minimal coincidence surfaces.
 The corresponding Demoulin surfaces or projective 
 minimal coincidence surfaces have the same projective metric 
 $8 bc \>dx dy$ but the different conformal classes of cubic forms
 $\lambda^{-3} b \>dx^3 + \lambda^3 c \>dy^3$. Moreover, 
 the functions $P$ and $Q$ change as $\lambda^{-2} P$ 
 and $\lambda^2 Q$, respectively.
\end{Remark}
\begin{Corollary}\label{coro:Demoulin}
 Retaining the assumptions in Theorem \ref{thm:Demoulin}, 
 the following are equivalent:
\begin{enumerate}
\item The surface $S$ is a Demoulin surface.
\item The first-order Gauss map $\gf$ 
 is a conformal Lorentz harmonic map into $\Q$.
\end{enumerate}
\end{Corollary}
\begin{proof}
 From Proposition \ref{prop:conformality}, it is 
 easy to see that the first-order Gauss map 
 is conformal if and only if it satisfies that $P=Q=0$, that is, 
 the surface is a Demoulin surface. Moreover, from Theorem 
 \ref{thm:Demoulin} the Gauss map of the Demoulin surface 
 is Lorentz harmonic.
\end{proof}
 Let $S$ be a Demoulin surface or projective 
 minimal coincidence surface and $F^{\lambda}$ 
 a family of frames 
 such that $(F^{\lambda})^{-1} d F^{\lambda} = 
 \alpha^{\lambda}$. 
 Then $F^{\lambda}$ will be called the {\it extended 
 Wilczynski frame} for a Demoulin surface or projective 
 minimal coincidence surface.

 We now show that the extended Wilczynski frame for a Demoulin surface 
 has an additional 
 order three cyclic symmetry.
 Let $\sigma$ be an order three automorphism on the 
 complexification of $\SL$ as follows:
$$
 \sigma X = \ad (E) X, \;\; X \in \SLC,
$$
 where $E =\di (1, \epsilon^2, \epsilon, 1)$ with 
 $\epsilon = e^{2\pi i/3 }$. Then it is easy to see that 
 $F(\lambda) (:= F^{\lambda})$ satisfies the symmetry
 $\sigma F(\lambda) =F(\epsilon \lambda)$,
 since $U(\lambda)(: = U^{\lambda})$ 
 and $V(\lambda)(:=V^{\lambda})$ satisfy 
 the same symmetry. It is also easy to see that 
 $\tau_1$ and $\sigma$ commute, and $\kappa = \tau_1 \circ \sigma$
 defines an order six automorphism. Thus, the extended Wilczynski 
 frame $F(\lambda)$ satisfies the symmetry
$$
 \kappa F(\lambda) = F(- \epsilon \lambda).
$$
 Note that $-\epsilon$ is the $6$th root of unity. 
 From the above argument, it is easy to see 
 that the extended Wilczynski frame 
 $F(\lambda)$ for a Demoulin surface 
 is an element of the twisted loop group of $\SL$:
 $$
 \LSLT = \{ g : \mathbb R^{\times} \to \SL\;|\; 
 \kappa g(\lambda) = g (-\epsilon \lambda)\}.
 $$
\begin{Theorem}\label{coro:primitivity}
 The first-order Gauss map of a Demoulin surface, which 
 is conformal Lorentz harmonic in $\Q = \GKOT$, 
 can be obtained by the natural projection 
 of a Lorentz primitive map into the $6$-symmetric space $\GKS$
 with  $K =\{\di (k_1, k_2, k_2^{-1}, k_1^{-1}) \;|\;k_1, k_2 \in \mathbb C^{\times}\}$.
\end{Theorem}
\begin{proof}
 The $0$th-eigenspace and $\pm1$st-eigenspaces of 
 the derivative of the order six automorphism 
 $\kappa =\tau_1 \circ \sigma$
 are described as follows:
\begin{equation*}
 \mathfrak g_{0} = \left\{\di(a_{11}, a_{22}, - a_{22}, -a_{11})
 \;|\;a_{ij} \in \C \right\},
\end{equation*}
and 
\begin{equation*}
 \mathfrak g_{-1} = \left\{
\left.
 \begin{pmatrix} 
0 & 0 & a_{13} & 0  \\ 
a_{21} & 0 & 0 & a_{13}  \\ 
0 & a_{32} & 0& 0  \\ 
0 & 0 & a_{21} & 0  
\end{pmatrix}
\right| a_{i j} \in \C
\right\},\;\;
 \mathfrak g_{1} =
\left\{
\left.
 \begin{pmatrix} 
0 & a_{12} & 0 & 0  \\ 
0 & 0 & a_{23} & 0  \\ 
a_{31} & 0 & 0& a_{12}  \\ 
0 & a_{31} & 0 & 0  
\end{pmatrix}
\right| a_{i j} \in \C
\right\}.
\end{equation*}
 From the matrices $U^{\lambda}$ and $V^{\lambda}$
 in \eqref{eq:extendedUV2} with $P=Q=0$, we see that 
 the conditions in \eqref{eq:Lorentzprimitivedef} 
 of a Lorentz primitive map are satisfied.
 The stabilizer of $\kappa$ is 
 $$
 K =\{\di (k_1, k_2, k_2^{-1}, k_1^{-1}) \;|\;k_1, k_2 \in \mathbb C^{\times}\}.
 $$
 Therefore there is a Lorentz primitive map $g= F J F^t$ with 
 $J=E J_1$ into the $6$-symmetric space $\GKS$ such that 
 $\pi \circ g = g_1$, where $\pi$ 
 is the natural projection $\pi :\GKS \to  \GKOTC$.
 We note that the projection $\pi$ of 
 a general Lorentz primitive map into $\GKS$
 is a harmonic map into $\GKOTC$ not $\GKOT$.
 However, the Lorentz primitive map $g$ induced from the 
 first-order Gauss map $g_1$ has an additional real
 structure.
 The eigenspaces $\mathfrak g_0$ and $\mathfrak g_{\pm 1}$ can 
 be decomposed into the real and the imaginary 
 parts, that is 
 $$
 \mathfrak g_{0} = \mathfrak g_{0}^{{\rm Re}} 
 \oplus \mathfrak g_{0}^{{\rm Im}}\;\;\mbox{and}\;\;
 \mathfrak g_{\pm 1} = \mathfrak g_{\pm 1}^{{\rm Re}} 
 \oplus \mathfrak g_{\pm 1}^{{\rm Im}},
 $$  
 where $\mathfrak g_{0}^{{\rm Re}}, \mathfrak g_{\pm 1}^{{\rm Re}}$ and 
 $\mathfrak g_{0}^{{\rm Im}}, \mathfrak g_{\pm 1}^{{\rm Im}}$ consist of 
 real and imaginary entries, respectively.
 Since the first order Gauss map takes 
 values in $\Q=\GKOT$, 
 the $(1,0)$- and $(0, 1)$-parts $\alpha^{\prime}$ 
 and $\alpha^{\prime \prime}$ for the Maurer-Cartan 
 form of the map $g_1$ take
 values in $ \mathfrak g_{0}^{{\rm Re}}+\mathfrak g_{-1 }^{{\rm Re}}$ 
 and $\mathfrak g_{0}^{{\rm Re}} +\mathfrak g_{1 }^{\rm Re}$, 
 respectively.
 Therefore, the projection $\pi$ combined with $g$
 gives a Lorentz harmonic map into $\GKOT$.
\end{proof}

\begin{Remark}
 Since we obtained 
 the Lorentz primitive map into the $6$-symmetric space 
 for a Demoulin surface, 
 the generalized Weierstrass type representation as in 
 \cite{DMPW} can be established. 
\end{Remark}
\appendix
\section{Projective minimal surfaces and the conformal Gauss maps}
 \label{ap:confGauss}

\subsection{The conformal Gauss map}
 We define a map $\gs$ by 
 $$
 \gs = F \js F^t, 
 $$
 where $\js = \offdi (1, -1, -1, 1)$.
 Similar to $\gf$,  it is easy to see that $\gs$ maps 
 into the space of quadrics with signature $(2, 2)$ which is
 isomorphic to the $\Q$.
 In fact the special linear group 
 $\SL$ transitively acts on this space by $g P g^t$ 
 with $g \in \SL$ and $P\in \Q$. 
 Then  the point stabilizer at $\js \in \Q$ is given by
 $K_2 = \{ X \in \SL \;|\; X \js X^t = \js\}$, which 
 is also isomorphic to ${\rm SO}_{2, 2}$:
 \begin{equation}\label{eq:Gaussmap2}
 \gs : M \to \Q \cong \GKO = \GKSO.
 \end{equation}
 This map $\gs$ is known to be a Lie quadric which has 
 the second order contact to the surface, 
 see \cite[Section 18]{Lane}.
 We call $\gs$ the {\it conformal Gauss map} for a surface $S$
 in $\RP$, see \cite{Th, BHJ}. 
 In \cite{MN}, the conformal Gauss map $\gs$ was called 
 the projective Gauss map. 
%
%\begin{Remark}
% The both matrices $\jf$ and $\js$ define quadratic forms
% on $\mathbb R^4$ with signature $(2, 2)$, however, 
% the orders are different. Moreover as mentioned above, 
% $\gf$ has only the first order contact, meanwhile 
% $\gs$ has the second order contact to the surface, 
% see \cite{Lane}.
%\end{Remark}
%
\begin{Proposition}[Theorem 3 in \cite{BHJ}]\label{prop:conformality2}
 The conformal Gauss map $\gs$ is conformal map.
\end{Proposition}
\begin{proof}
 We introduce the inner product on the tangent space of $\Q$ 
 as in the proof of Proposition \ref{prop:conformality}.
 Then a direct computation shows that
$$
 {\gs}_x = 2 F \di (b Q, 0, -b, 0) F^t
 \;\;\mbox{and}\;\;
 {\gs}_y = 2 F \di (c P, -c, 0, 0) F^t.
$$
 Thus 
$$
\langle {\gs}_x, {\gs}_x \rangle = \langle  {\gs}_y,  {\gs}_y \rangle  = 0
\;\; \mbox{and} \;\;
\langle  {\gs}_x,  {\gs}_y \rangle = \langle {\gs}_y, {\gs}_x \rangle = 4 b c \neq  0.
$$
 Since the coordinates $(x, y)$ are null for the conformal 
 structure induced by $S$, 
 the conformal Gauss map $\gs$ is conformal.
\end{proof}

\subsection{Projective minimal surfaces and the conformal Gauss maps}
 Let $\tau_2$ be the outer involution on 
 $\SL$ associated to the symmetric 
 space $\Q$ in \eqref{eq:Gaussmap2} 
 defined by $\tau_2 (X) = \js X^{t-1}\js , \; X \in \SL$
 and $\js =\offdi (1, -1, -1, 1)$. 
 Abuse of notation, we denote the 
 differential of $\tau_2$ by the same symbol $\tau_2$  
 which is an outer involution on $\sl$:
\begin{equation}\label{eq:tau1}
\tau_2 (X) = - \js X^t \js, \; \;X \in \sl.
\end{equation}
 Let us consider the eigenspace decomposition
 of $\mathfrak g = \sl$ with respect to $\tau_2$, that 
 is, $\mathfrak g = \mathfrak k_2 \oplus \mathfrak p_2$,
 where $\mathfrak k_2$ is the $0$th-eigenspace 
 and $\mathfrak p_2$ is the $1$st-eigenspace as follows:
$$
 \mathfrak k_2 =\left\{\left.
\begin{pmatrix}
 a_{11} & a_{12} & a_{13} & 0\\
 a_{21} & a_{22} & 0 &  a_{13}\\
 a_{31} & 0 & -a_{22} & a_{12}\\
 0 & a_{31} & a_{21} & -a_{11}
\end{pmatrix}
\right| \;a_{ij} \in \R
\right\}, \;\; 
 \mathfrak p_2 =
\left\{\left.
\begin{pmatrix}
 a_{11} & a_{12} & a_{13} & a_{14}\\
 a_{21} & -a_{11} & a_{23} &  -a_{13}\\
 a_{31} & a_{32} & -a_{11} & -a_{12}\\
 a_{41} & -a_{31} & -a_{21} & a_{11}
\end{pmatrix}
\right| \;a_{ij} \in \R
\right\}.
$$
 According to this decomposition  
 $\mathfrak g = \mathfrak k_2 \oplus \mathfrak p_2$, 
 the Maurer-Cartan form $\alpha = F^{-1} d F = U dx + Vdy$
 can be decomposed into  
 $$
 \alpha = \alpha_{\mathfrak k_2} + \alpha_{\mathfrak p_2} =
 U_{\mathfrak k_2}dx + V_{\mathfrak k_2}dy + U_{\mathfrak p_2}dx + 
 V_{\mathfrak p_2}dy,
 $$
 where $U = U_{\mathfrak k_2} + U_{\mathfrak p_2}$ and
 $V = V_{\mathfrak k_2} + V_{\mathfrak p_2}$.
 Let us insert the parameter $\lambda \in \mathbb R^{\times}$ 
 into $U$ and $V$ as follows:
 $$
 U^{\lambda}= U_{\mathfrak k_2} + \lambda^{-1} U_{\mathfrak p_2} \;\;\mbox{and}  \;\;
 V^{\lambda} = V_{\mathfrak k_2} + \lambda V_{\mathfrak p_2}.
 $$
 Then a family of $1$-forms $\alpha^{\lambda}$ 
 is defined as follows:
\begin{equation}\label{eq:alphalambda1} 
 \alpha^{\lambda} = 
 \alpha_{\mathfrak k_2} + \lambda^{-1} \alpha_{\mathfrak p_2}^{\prime}
 + \lambda \alpha_{\mathfrak p_2}^{\prime \prime} 
 = U^{\lambda} dx + V^{\lambda} dy.
\end{equation}
 In fact the matrices $U^{\lambda}$ 
 and $V^{\lambda}$ are explicitly given 
 as follows:
\begin{equation}\label{eq:extendedUV1}
 U^{\lambda}= 
 \begin{pmatrix}
 \frac{c_x}{2 c} & P & k  &  \lambda^{-1}b Q  \\[0.1cm]
 1  & -\frac{c_x}{2 c} & 0 & k \\[0.1cm]
 0&\lambda^{-1} b&\frac{c_x}{2 c} & P \\[0.1cm]
 0 & 0 & 1 &-\frac{c_x}{2 c}
 \end{pmatrix}, \;\;
 V^{\lambda}= 
 \begin{pmatrix}
 \frac{b_y}{2 b} & \ell &  Q  & \lambda c P \\[0.1cm]
 0  & \frac{b_y}{2 b} & \lambda c & Q \\[0.1cm]
 1&0& -\frac{b_y}{2 b} & \ell \\[0.1cm]
 0 & 1 & 0 &-\frac{b_y}{2 b}
 \end{pmatrix}.
\end{equation}
 Then the projective minimal surface can be characterized 
 by the Lorentz harmonicity of the conformal Gauss map 
 \cite{Th}, \cite[Theorem 7]{BHJ},
 and by a family of flat connections.
\begin{Theorem}[\cite{Th}, Theorem 7 in \cite{BHJ}]
 \label{thm:projmini}
 Let $S$ be a surface in $\RP$ and $\gs$ the 
 conformal Gauss map defined  in \eqref{eq:Gaussmap2}.
 Moreover, let $\alpha^{\lambda}\> (\lambda
 \in \R^{\times})$  be a family of $1$-forms defined in 
 \eqref{eq:alphalambda1}. Then 
 the following are mutually equivalent:
\begin{enumerate}
\item The surface $S$ is a projective minimal surface. 
\item The conformal Gauss map $\gs$ 
 is a conformal Lorentz harmonic map into $\Q$.
\item $d + \alpha^{\lambda}$ is a family of flat connections 
 on $\D \times \SL$.
\end{enumerate}
\end{Theorem}
\begin{proof}
 Let us compute the flatness conditions of $d + \alpha^{\lambda}$, 
 that is, the Maurer-Cartan equation $ d \alpha^{\lambda} + 
 \tfrac{1}{2}[\alpha^{\lambda} \wedge \alpha^{\lambda}] =0$. 
 It is easy to see that except the $(1, 4)$-entry, the 
 Maurer-Cartan equation is equivalent to \eqref{eq:comp1}.
 Moreover, the $\lambda^{-1}$-term and the $\lambda$-term 
 of the $(1, 4)$-entry are equivalent to 
 that the first equation and the second equation 
 in \eqref{eq:projmin}, respectively. 
 Thus the quivalence of $(1)$ and $(3)$ follows.
 
 The equivalence of $(2)$ and $(3)$ follows from Theorem
 \ref{thm:harmonicflat}, since the family of 
 $1$-forms $\alpha^{\lambda}$ is given 
 by the involution $\tau_2$ and it defines 
 the symmetric space $\Q = \GKO$.
\end{proof}
 The above theorem implies that if $S$ is a projective 
 minimal surface, then there exists a family of projective 
 minimal surface $S^{\lambda}\> (\lambda \in \mathbb R^{\times})$
 such that $S^{\lambda}|_{\lambda =1}=S$. Projective minimal surfaces of 
 the family have the same projective metric $8bc \> dx dy$
 but the different conformal classes of 
 cubic forms $\lambda^{-1} b \>dx^3 
 + \lambda c \>dy^3$. 
 Thus the family of the Maurer-Cartan form $\alpha^{\lambda}$
 defines a family of Wilczynski frames $F^{\lambda}$ such 
 that $(F^{\lambda})^{-1} d F^{\lambda} = \alpha^{\lambda}$. 
 It is easy to see that $F^{\lambda}$ is an element of 
 the twisted loop group of $\SL$:
 $$
 \LSLO = \{ g : \mathbb R^{\times} \to \SL\;|\; 
 \tau_2 g(\lambda) = g (-\lambda)\}.
 $$
 This family of Wilczynski frames $F^{\lambda}$ will be called 
 the {\it extended Wilczynski frame} for a projective minimal  
 surface.

\bibliographystyle{plain}
\def\cprime{$'$}

\end{document}